\documentclass[review]{elsarticle}

\usepackage{algorithm}
\usepackage{algorithmic}
\usepackage{graphicx}
\usepackage{epstopdf}
\usepackage{diagbox}
\usepackage{amsthm,amsfonts,amssymb}

\usepackage{amsmath}
\usepackage{amssymb}
\usepackage{url}
\usepackage{amsthm,amsfonts,amssymb}

\newtheorem{theorem}{Theorem}
\newtheorem{definition}{Definition}
\newtheorem{lemma}{Lemma}
\newtheorem{conjecture}{Conjecture}

\journal{Journal of \LaTeX\ Templates}

\begin{document}

\begin{frontmatter}



    \title{Linear convergence of the Collatz method  for computing the Perron eigenpair
of primitive dual number matrix\tnoteref{mytitlenote}}
    \tnotetext[mytitlenote]{This work is supported by the National Natural Science Foundation of China (Grant No. 12171271).}


    \author{Yongjun Chen }
    \author{\qquad Liping Zhang\corref{mycorrespondingauthor}}
    \cortext[mycorrespondingauthor]{Corresponding author. \emph{Email address}: ~~lipingzhang@mail.tsinghua.edu.cn}
    \address{Department of Mathematical Sciences, Tsinghua University, Beijing 100084, China}

\begin{abstract}
Very recently, Qi and Cui extended the Perron-Frobenius theory to dual number matrices with
primitive and irreducible nonnegative standard parts and proved that they have Perron eigenpair and Perron-Frobenius eigenpair. The Collatz method was also extended to find Perron eigenpair. Qi and Cui proposed two conjectures. One is the $k$-order power of a dual number matrix tends to zero if and only if the spectral radius of its standard part less than one, and another is the linear convergence of the Collatz method. In this paper, we confirm these conjectures and provide theoretical proof. The main contribution is to show that the Collatz method R-linearly converges with an explicit rate.
\end{abstract}



\begin{keyword}
Dual numbers, eigenvalues, dual primitive matrices, irreducible nonnegative
matrices, Collatz method, linear convergence

\MSC[2022] 15A03 \sep 49M29 \sep 65K05
\end{keyword}

\end{frontmatter}


\section{Introduction}\label{sec:introduction}
Dual numbers and dual number matrices have applications in dynamic analysis of spatial mechanisms \cite{qc10,qc13,qc17,qc18}, robotics \cite{qc5}, kinematic synthesis \cite{qc1,qc9}, Markov chain process \cite{qcnew}, and hence have attracted
much attention. Very recently, the eigenvalue theory has become a  research focus of dual complex matrix analysis \cite{qcnew, qc14,qc12,qc11}.

It is easy to see that dual numbers are special cases of dual complex numbers. Dual complex numbers was considered as a special case of dual quaternion numbers in \cite{qc11} and then eigenvalues of dual complex matrices were defined \cite{qc14}, which were further studied in \cite{qc12}. It
was shown that an $n\times n$ dual complex matrix has exactly $n$ eigenvalues with
$n$ appreciably linearly independent eigenvectors if and only if it is similar to
a diagonal matrix \cite{qc12}. Motivated by considering both probabilities and perturbation, or error bounds, or variances, in the Markov chain process, Qi and Cui \cite{qcnew} investigated eigenvalues of dual
number matrices with primitive and irreducible nonnegative standard parts following the approach of regarding dual complex numbers as a special case of dual quaternion numbers \cite{qc11}. They extended the Perron-Frobenius theory to dual number matrices with
primitive and irreducible nonnegative standard parts, and proved that such a dual number matrix always has a positive dual number eigenvalue with a positive dual number eigenvector, which is called Perron eigenvalue. The standard part of
Perron eigenvalue is larger than or equal to the modulus of the standard
part of any other eigenvalue of this dual number matrix, and the Collatz minimax theorem also holds for Perron eigenvalue. Based on the Collatz minimax theorem, the Collatz method was extended to compute Perron eigenpair of a primitive dual number matrix.

Qi and Cui proposed two conjectures in \cite{qcnew}. One is that $A^k$ with an $n\times n$ dual number matrix $A=A_s+A_d\varepsilon$ tends to zero matrix as the integer $k$ tends to $\infty$ if and only if the spectral radius of $A_s$ less than $1$ . They confirmed that this is true if $A_sA_d=A_dA_s$, but the general case is not obtained in \cite{qcnew}. Another is that the Collatz method is linearly convergent. This is observed numerically from Figure 2 in \cite{qcnew}, but without theoretical proof.

In this paper, we give positive answers for the proposed two conjectures in \cite{qcnew}. We use Jordan decomposition of real matrix and show that for a general $n\times n$ dual number matrix $A=A_s+A_d\varepsilon$, the necessary and sufficient condition for the convergence of $A^k$ to zero matrix $O$ is the spectral radius of the standard part $A_s$ less than $1$. Along with the proof line of this result, we establish an explicit linear convergence rate of the proposed Collatz method \cite[Algorithm 1]{qcnew} for computing Perron eigenpairs of dual number matrices with primitive and irreducible nonnegative standard parts.

This paper is  arranged as follows. In Section \ref{Preliminaries}, we review some preliminary definitions and results for dual numbers, dual complex numbers, eigenvalues of dual complex matrices, and primitive or irreducible nonnegative dual number matrix.
 In Section \ref{pow}, we give a necessary and sufficient condition for the convergence of the power of dual number matrix. In Section \ref{Colmet}, we recall the Collatz method for computing Perron eigenpair of a dual number matrix with primitive or irreducible nonnegative standard part and  establish an explicit linear convergence rate for this Collatz method. Finally, some conclusions are given in Section \ref{conclusion}.

\section{Preliminaries}\label{Preliminaries}
In this section, we review some preliminary knowledge of dual numbers, dual
complex numbers, eigenvalues of dual complex matrices \cite{qc14}, eigenvalues of primitive dual number matrices \cite{qcnew}, and eigenvalues of dual number matrices with irreducible nonnegative standard parts \cite{qcnew}.

\subsection{Dual Numbers and Dual Complex Numbers}

The set of positive integers is denoted by $\mathbb{N}$. We denote $\mathbb{R}$ and $\mathbb{C}$ as real number field and complex number field, respectively. Let $\mathbb{D}$, $\mathbb{D}_+$, and $\mathbb{D}_{++}$ respectively denote the sets of dual numbers,  nonnegative dual numbers, and positive dual numbers. Throughout of this paper, when a dual number is nonnegative or positive, we say it is a nonnegative
dual number or a positive dual number, respectively. If we say a number
is a nonnegative number or positive number, then that number should be a real
number.

The set of dual complex numbers is denoted by $\mathbb{DC}$. Let the symbol $\varepsilon$ denote the infinitesimal unit which is commutative with complex numbers and satisfies $\varepsilon\neq 0,\varepsilon^{2}=0$.

A number is denoted by a lower-case letter, a vector is denoted by a bold lower-case letter, and a matrix is denoted by a capital letter. We use $0$, $\bf{0}$, and $O$ to denote zero number, zero vector, and zero matrix,  respectively. For any positive integer $n$, $I_n$ denotes $n\times n$ unit matrix. For $a\in\mathbb{C}$, we respectively use $a^*$ and $|a|$ to denote its conjugate and magnitude. For $A=(a_{ij})\in\mathbb{R}^{n\times n}$, define $\|A\|_1=\sum\limits_{i=1}^n\sum\limits_{j=1}^n|a_{ij}|$.

\begin{definition}
A \textbf{dual complex number} $a=a_{s}+a_{d}\varepsilon\in \mathbb{DC}$ has standard part
$a_{s}\in \mathbb{C}$ and dual part $a_{d}\in \mathbb{C}$. If $a_{s}\neq 0$, then we say that $a$ is \textbf{appreciable}. If $a_{s},a_{d}\in \mathbb{R}$, then $a$ is called a \textbf{dual number}, i.e., $a\in\mathbb{D}$.
\end{definition}

The following definition lists some operators about dual complex numbers; see, e.g., \cite{qcnew,qc13}.

\begin{definition}\label{def1}
Let $a=a_{s}+a_{d}\varepsilon$ and $b=b_{s}+b_{d}\varepsilon$ be any two {\bf dual complex numbers}. We give the conjugate, magnitude of $a$, and the sum, product between $a$ and $b$.
\begin{itemize}
\item[{\rm (i)}] The {\bf conjugate} of $a$ is
$a^{\ast }=a_{s}^{\ast}+a_{d}^{\ast }\varepsilon,$ where $a_{s}^{\ast}$ and $a_{d}^{\ast}$ are conjugates of $a_s,a_d$ respectively.

\item[{\rm (ii)}] The {\bf sum} and {\bf product} of $a$ and $b$ are
$$ a+b=b+a=\left (a_{s}+b_{s} \right )+\left (a_{d}+b_{d} \right )\varepsilon,$$
and
$$ab=ba=a_{s}b_{s} +\left (a_{s}b_{d}+a_{d}b_{s} \right )\varepsilon.$$
Clearly, the add and multiplication  operators of dual complex numbers are all commutative.

\item[{\rm (iii)}] The {\bf magnitude} of $a$ is
$$\left | a\right |=\begin{cases}
\left | a_{s}\right |+\dfrac{a_{s}^*a_d}{|a_{s}|}\varepsilon, & \text{{\rm if} $a_{s}\neq 0$}, \\
\left | a_{d} \right |\varepsilon, & \text{{\rm if} $a_{s}=0$}.
\end{cases}$$
\end{itemize}
\end{definition}

For dual numbers, we can define the order and division operations \cite{qc13}.
\begin{definition}\label{def2}
Let $a=a_{s}+a_{d}\varepsilon$ and $b=b_{s}+b_{d}\varepsilon$ be any two {\bf dual numbers}. We give the magnitude of $a$, and the order and division operations between $a$ and $b$.
\begin{itemize}
\item[{\rm (i)}] We say that $a > b$ if
$$\text{$a_{s} > b_{s}$ ~~~{\rm or}~~~ $a_{s}=b_{s}$ {\rm and} $a_{d} > b_{d}$}.$$
Moreover, the sets $\mathbb{D}_+$ and $\mathbb{D}_{++}$ can be rewritten as
\begin{equation*}
\begin{aligned}
\mathbb{D}_+&=\{x=x_s+x_d\varepsilon\in\mathbb{D}:~~ \text{$x\ge 0$, {\rm i.e.}, $x_s\ge 0$ {\rm or} $x_s=0,x_d\ge 0$}\},\\
 \mathbb{D}_{++}&=\{x=x_s+x_d\varepsilon\in\mathbb{D}:~~ \text{$x> 0$, {\rm i.e.}, $x_s> 0$ {\rm or} $x_s=0,x_d> 0$}\}.
 \end{aligned}
\end{equation*}
\item[{\rm (ii)}] The {\bf magnitude} of $a$ is
$$\left | a\right |=\begin{cases}
\left | a_{s}\right |+{\rm sgn}( a_{s}) a_{d}\varepsilon, & {\rm if} \quad a_{s}\neq 0, \\
\left | a_{d} \right |\varepsilon, & \text{\rm otherwise.}
\end{cases}$$
Clearly, $|a|\in\mathbb{D}_+$, i.e., the magnitude of a dual number is a nonnegative dual number.

\item[{\rm (iii)}] When $b_s\ne 0$ or $a_s=b_s=0$, we can define the division operation of dual numbers as
$$\frac{ a_{s}+ a_{d}\varepsilon}{b_{s}+ b_{d}\varepsilon}=\begin{cases}
\dfrac{ a_{s}}{b_{s}}+\left (\dfrac{ a_{d}}{b_{s}}-\dfrac{ a_{s}b_{d}}{b_{s}b_{s}} \right )\varepsilon, & {\rm if} \quad b_{s}\neq 0, \\
\dfrac{ a_{d}}{b_{d}}+c\varepsilon, & {\rm if} \quad a_{s}=b_{s}=0,
\end{cases}
$$
where $c$ is an arbitrary real number.
\end{itemize}
\end{definition}

Following the above definitions, some related preliminary results of dual complex number vectors are similarly given.  A dual complex number vector is denoted by $\mathbf x=\mathbf{x}_s+\mathbf{x}_d\varepsilon=\left ( x_{1},x_{2},\cdots ,x_{n}\right )^{T}\in \mathbb{DC}^{n}$.
If $\mathbf x_{s}\neq \mathbf 0 $, then we say that $\mathbf x$ is appreciable.
The $2$-norm of $\mathbf{x}$ is defined as
$$\left \| \mathbf x \right \| _{2}=
\begin{cases}
\left \| \mathbf x_{s} \right \| _{2}+ \dfrac{\mathbf x_{s}^{\ast }\mathbf x_{d}}{\left \| \mathbf x_{s} \right \| _{2}}\varepsilon, & \text{if} \quad  \mathbf x_{s}\neq  \mathbf 0,\\
\left \| \mathbf x_{d} \right \| _{2} \varepsilon, & \text{if} \quad \mathbf x_{s}\neq  \mathbf 0.
\end{cases}$$
The dual complex number vector $\mathbf x$ is called a unit vector if $\left \| \mathbf x \right \| _{2}=1$, or equivalently,
$\left \| \mathbf x_{s} \right \| _{2}=1$ and $ \mathbf x_{s}^{\ast } \mathbf x_{d}=0$.  We use $\mathbf e_{1},\cdots ,\mathbf e_{n}$ to denote the unit vectors in $\mathbb{R}^{n}$. Clearly, they are also unit vectors of $\mathbb{DC}^{n}$.

Following \cite[Theorem 3.3]{qc4}, the normalization of $\mathbf x$ is defined as
$$\mathbf{y}=\dfrac{\mathbf x}{\left \| \mathbf x \right \|_{2}}=\mathbf{y}_s+\mathbf{y}_d\varepsilon,$$
where
$$ \mathbf{y}_s=\dfrac{\mathbf x_{s}}{\left \| \mathbf x_{s} \right \| _{2}}, \quad \mathbf{y}_d=
\dfrac{ \mathbf x_{d}}{\left \| \mathbf x_{s} \right \| _{2}}-\dfrac{ \mathbf x_{s}}{\left \| \mathbf x_{s} \right \| _{2}}\dfrac{\mathbf x_{s}^{\ast }\mathbf x_{d}}{\left \| \mathbf x_{s} \right \| _{2}^{2}}
$$
if $\mathbf x_{s} \neq 0$;  otherwise,
$$\mathbf{y}_s=\dfrac{ \mathbf x_{d}}{\left \| \mathbf x_{d} \right \| _{2}},\quad  \text{$\mathbf y_{d}$ is any complex vector satisfying $\mathbf y_{s}^{\ast }\mathbf y_{d}=0$}.$$

 An $n\times n$ dual complex number matrix is denoted by $A=A_{s}+A_{d}\varepsilon\in \mathbb{DC}^{n\times n}$. We say that $A$ is invertible if $A_{s}$ is invertible. By \cite{qc12}, if $A$ is invertible, then  $A^{-1}=A_{s}^{-1}-A_{s}^{-1}A_{d}A_{s}^{-1}\varepsilon$. In addition, the
$F^{R}$-norm of $A$ is defined as $\left \| A\right \|_{F^{R}}=\sqrt{\left \| A_{s}\right \|_{F^{R}}+\left \| A_{d}\right \|_{F^{R}}}$. It is easy to see that $A$ tends to $O$ if and only if its every element tends to $0$.

\subsection{Perron Eigenpair of Primitive Dual Number Matrices}

Eigenvalues of dual complex matrices were introduced in \cite{qc14} and studied in
details in \cite{qc12}.

\begin{definition}[Eignvalues of Dual Complex Number Matrices] Let  $A=A_s+A_d\varepsilon\in \mathbb{DC}^{n\times n}$. If there exist a dual complex number $\lambda_s+\lambda_d\varepsilon$ and an appreciable dual complex number vector $\mathbf{x}=\mathbf{x}_s+\mathbf{x}_d\varepsilon\in \mathbb{DC}^n$, such that
\begin{equation}\label{eigenvalue}
A\mathbf{x}=\lambda \mathbf{x},
\end{equation}
then $\lambda$ is called an {\bf eigenvalue} of $A$ with an {\bf eigenvector} $\mathbf{x}$.
\end{definition}

Since $\varepsilon^2=0$, it follows from (\ref{eigenvalue}) that
\begin{equation}\label{eigenvalue1}
A_{s}\mathbf{x}_{s}=\lambda_{s} \mathbf x_{s}
\end{equation}
and
\begin{equation}\label{eigenvalue2}
\left (A -\lambda_{s}I\right )\mathbf x_{d}-\lambda_{d} \mathbf x_{s}=-A_{d}\mathbf x_{s}.
\end{equation}
Since $\mathbf{x}$ is appreciable in (\ref{eigenvalue}), i.e., $\mathbf{x}_{s}\neq \mathbf{0}$, it follows from (\ref{eigenvalue1}) that $\lambda_s$ is an eigenvalue of the complex matrix $\mathbf{x}_s$ with an eigenvector $\mathbf{x}_s$. We use $\rho(A_s)$ to denote the spectral radius of $A_s$, then $|\lambda_s|\le \rho(A_s)$. By (\ref{eigenvalue2}) and Definition \ref{def1} (iii), we see that for any eigenvalue
$\lambda$ of $A$ and any positive number $\delta$, it holds
\begin{equation}\label{eigeq1}
\rho(A_s)>|\lambda|-\delta.
\end{equation}

We say that  $A=A_{s}+A_{d}\varepsilon$ is an $n\times n$ {\bf dual number matrix} if $A_s,A_d$ are all $n\times n$ real matrix, i.e., $A_s,A_d\in\mathbb{R}^{n\times n}$. The set of $n\times n$ dual number matrices is denoted by $\mathbb{D}^{n\times n}$. We say that $A\in \mathbb{D}^{n\times n}$ with a  positive  standard part (i.e., $A_s\in \mathbb{R}_{++}^{n\times n}$ or  $A_s>O$) is a {\bf positive  dual number matrix}. As shown in \cite{qc12,qc13,qcnew}, a square dual number matrix or even positive dual number matrix may have no eigenvalue at all.  So,  Qi and Cui \cite{qcnew} considered the eigenvalue of a dual number matrix when its
standard part is primitive or irreducible nonnegative. Some results on primitive matrices and irreducible nonnegative matrices can be referred in \cite{qc2,qc16}.  We now list some definitions and results on eigenvalues of dual number matrices with
  primitive or irreducible nonnegative standard parts, which were given in \cite{qcnew}.

\begin{definition}[Primitive Dual Number Matrix]
We say that  $A=A_{s}+A_{d}\varepsilon\in \mathbb{D}^{n\times n}$ is a {\bf primitive dual number matrix}, if
$A_{s}$ is a primitive matrix, i.e., $A_{s}\geq O$, and there exists a positive integer $k$ such that $A_s^{k}$ is positive, i.e., $A^k_{s}> O$.
\end{definition}

\begin{definition}[Irreducible Nonnegative Dual Number Matrix]
We say that $A=A_{s}+A_{d}\varepsilon\in \mathbb{D}^{n\times n}$ is an {\bf irreducible nonnegative dual number matrix}, if $A_{s}$ is an irreducible nonnegative matrix, i.e., $A_s\ge O$ and it is not permutation-similar to any block-triangular matrix.
\end{definition}

\begin{theorem}\label{lemma2}
Suppose that  $A=A_{s}+A_{d}\varepsilon$ is an $n\times n$ primitive dual number matrix, then there exists an eigenvalue $\lambda=\lambda_{s}+\lambda_{d}\varepsilon\in \mathbb{D}_{++}$ of  $A$ with corresponding eigenvector
$\mathbf{x}=\mathbf{x}_{s}+\mathbf{x}_{d}\varepsilon\in \mathbb{D}^n_{++}$ such that
\begin{itemize}
\item[{\rm (i)}] $\lambda_{s}$ is the Perron eigenvalue of $A_{s}$ with multiplicity $1$, $\mathbf{x}_{s}$ is the Perron eigenvector of $A_{s}$, and for any other eigenvalue $\mu =\mu _{s}+\mu _{d}\varepsilon\in \mathbb{DC}$ of $A$ with corresponding eigenvector $\mathbf{z}=\mathbf{z}_{s}+\mathbf{z}_{d}\varepsilon\in \mathbb{DC}^n$, it holds $
\lambda _{s}>|\mu _{s}|$ and $\mathbf{z}_{s}$ cannot be a nonnegative vector;
\item[{\rm (ii)}] if $A_{d}$ is nonnegative, then $\lambda_{d}$ is nonnegative; if $A_{d}$ is nonnegative and nonzero, then $\lambda_{d}$ is positive. Moreover, \begin{equation}\label{eq5}
\lambda _{d}=\dfrac{\mathbf y_{s}^{T}A_{d}\mathbf x_{s}}{\mathbf y_{s}^{T}\mathbf x_{s}},
\end{equation}
where $\mathbf y_{s}$ is the Perron eigenvector of $A^T_s$ corresponding to $\lambda_s$ with $\mathbf{y}_s^T\mathbf{x}_s>0$;
\item[{\rm (iii)}] $\mathbf{x}_{d}$ is the solution of
\begin{equation}
\left ( A_{s}-\lambda _{s}I\right )\mathbf x_{d}=\left ( \lambda _{d}I-A_{d}\right )\mathbf x_{s}.
\end{equation}
If we require the 2-norm of $\mathbf x$ to be $1$, then $\mathbf x_{d}$ is unique.

\item[{\rm (iv)}] it holds for any $\mathbf{u}\in\mathbb{D}_{++}^n$ that
\begin{equation}
\underset{i}{\min}\dfrac{\left ( A\mathbf u\right )_{i}}{u_{i}}\leq \lambda \leq \underset{i}{\max}\dfrac{\left ( A\mathbf u\right )_{i}}{u_{i}}.
\end{equation}
Furthermore,
\begin{equation}
\underset{\mathbf{u}>\mathbf{0}}{\max}\,\underset{i}{\min}\dfrac{\left ( A\mathbf{u}\right )_{i}}{u_{i}}\leq \lambda \leq \underset{\mathbf{u} > \mathbf{0}}{\min}\,\underset{i}{\max}\dfrac{\left ( A\mathbf u\right )_{i}}{u_{i}}.
\end{equation}
\end{itemize}
\end{theorem}

We follow the names given in \cite{qcnew} and call $\lambda$ in Theorem \ref{lemma2}
the {\bf Perron eigenvalue} of $A$ and $\mathbf{x}$ the {\bf Perron eigenvector} or {\bf right Perron
eigenvector} of $A$. $\lambda$ is called  the {\bf spectral radius} of $A$, denoted as $\rho(A)=\lambda$.
The Perron eigenvector $\mathbf{y}=\mathbf{y}_s+\mathbf{y}_d\varepsilon$ of $A^T=A_s^T+A_d^T\varepsilon$
is called the {\bf left Perron eigenvector} of $A$. Moreover, $\mathbf{y}^T\mathbf{x}>0$.

By \cite[Theorem 5.1]{qcnew},  an $n\times n$ irreducible nonnegative dual number matrix $A=A_{s}+A_{d}\varepsilon$ has the {\bf Perron-Frobenius eigenvalue}, and there exists $\beta>0$ such that $(A_s+\beta I_n)+A_{d}\varepsilon$ is a primitive dual number matrix. So, the Collatz method given in \cite[Algorithm 1]{qcnew} also can compute the Perron-Frobenius eigenvalue of $A$. Hence, we only discuss the Collatz method for Perron eigenpairs of primitive dual number matrices in the sequel.

The numerical results given in \cite{qcnew} showed that the Collatz method has a rate of linear convergence, but without theoretical proof. So, two conjectures were proposed in \cite{qcnew}, one is the convergence of $k$-order power of a dual number matrix and the other is the linear convergence of the Collatz method.
In the next two sections, we give positive answers for this two conjectures. The proof approach of the first conjecture is the key to the proof of the second one.

\section{Convergence of $k$-Order Power of Dual Number Matrix}\label{pow}

The purpose of this section is to give a positive answer to the conjecture given in \cite{qcnew}, which is concerned with the convergence of $k$-order power of a dual number matrix.

\begin{conjecture}\label{cx1}
Let  $A=A_s+A_d\varepsilon$ be an $n\times n$ dual number matrix and $k\in \mathbb{N}$. Then
$$\lim_{k\to\infty}A^k=O$$
if and only if $\rho(A_s)<1$, where $\rho(A_s)$ is the spectral radius of $A_s$.
\end{conjecture}

Qi and Cui \cite{qcnew} confirmed that Conjecture \ref{cx1} is true if $A_sA_d=A_dA_s$, but the general case is not obtained in \cite{qcnew}.

As we know, if $T$ is an $n\times n$ complex matrix, then $T^k\to O$ as $k\to \infty$ if and only if $\rho \left ( T\right )< 1$. That is,
\begin{equation}\label{realmeq1}
\lim_{k\to\infty}T^k=O\quad \Longleftrightarrow\quad \rho(T)<1,\quad \forall ~T\in\mathbb{C}^{n\times n}.
\end{equation}
By this conclusion and Jordan decomposition of $A_{s}$, we show that Conjecture \ref{cx1} is true for general dual number matrices.

\begin{theorem}\label{thm2}
Let $A=A_{s}+A_{d}\varepsilon\in\mathbb{D}^{n\times n}$ and $k\in\mathbb{N}$. Then $A^{k}$ converges to zero matrix $O$ if and only if $\rho \left ( A_{s}\right )< 1$.
\end{theorem}
\begin{proof} We first show the sufficiency and suppose that $\rho \left ( A_{s}\right )< 1$.
Assume that the Jordan decomposition of $A_{s}$ is
$$A_{s}=P\text{diag}\left ( J_{1},J_{2},\cdots ,J_{t}\right )P^{-1},$$
where $J_{1},J_{2},\cdots ,J_{t}$ are Jordan blocks that belong to the eigenvalues $\lambda _{1},\lambda _{2},\cdots ,\lambda _{t}$, respectively. Denote $X=P^{-1}A_{s}P$ and $Y=P^{-1}A_{d}P$, then we have
\begin{equation}\label{dmeq1}
    \begin{aligned}
         P^{-1}A^{k}P&=\left (P^{-1}AP \right )^{k}=\left ( P^{-1}A_{s}P+P^{-1}A_{d}P\varepsilon \right )^{k}=\left ( X+Y\varepsilon \right )^{k}\\&=X^{k}+\left ( YX^{k-1}+XYX^{k-2}+X^2YX^{k-3}+\cdots +X^{k-1}Y\right )\varepsilon,
    \end{aligned}
\end{equation}
where the last equality holds due to $\varepsilon^2=0$.

We will consider the matrices in the form of $YX^{k-1}, XYX^{k-2}, X^2YX^{k-3},\cdots, X^{k-1}Y$ in (\ref{dmeq1}).
Denote $J_n(\lambda)$ as Jordan block belonging to the eigenvalue $\lambda$, i.e.,
$$
J_n(\lambda)=\begin{bmatrix}
\lambda & 1 & 0 & \cdots & 0 &0\\
0 & \lambda & 1 & \cdots & 0 &0\\
\vdots &\vdots & \vdots & \ddots & \vdots &\vdots\\
0 & 0 &0& \cdots & \lambda   &1\\
0 & 0 &0& \cdots & 0 & \lambda
\end{bmatrix}.
$$
Then, applying  the binomial theorem, for $k\in\mathbb{N}$ we have
\begin{equation}\label{dmeq2}
(J_n(\lambda))^k=\begin{bmatrix}
\lambda^k & k\lambda^{k-1} & \cdots & \mathrm{C}_k^{n-1}\lambda^{k-n+1}\\
0 & \lambda^k & \cdots & \mathrm{C}_k^{n-2}\lambda^{k-n+2}\\
\vdots &\vdots &  \ddots & \vdots \\
0 & 0 &\cdots &  \lambda^k
\end{bmatrix}.
\end{equation}
Denote $\rho \left ( A_{s}\right )=\lambda$ and $\alpha=\underset{1\le i,j\le n}{\max} \left \{ |y_{ij}| \right \}$ where $y_{ij}$ is the $(i,j)$-th element of $Y$. Define an $n\times n$ matrix $M$ by
$$
M=\begin{bmatrix}
    \alpha& \alpha &\cdots  &\alpha \\
    \alpha& \alpha &\cdots  &\alpha \\
     \vdots &\vdots &\cdots  &\vdots  \\
   \alpha& \alpha &\cdots  &\alpha
    \end{bmatrix}.
$$
Then, for any nonnegative integer $p,q$, it follows from (\ref{dmeq2}) that
\begin{equation*}
\begin{aligned}
M(J_n(\lambda))^q&=\begin{bmatrix}
    \alpha& \alpha &\cdots  &\alpha \\
    \alpha& \alpha &\cdots  &\alpha \\
     \vdots &\vdots &\cdots  &\vdots  \\
   \alpha& \alpha &\cdots  &\alpha
    \end{bmatrix}\begin{bmatrix}
    \lambda^q & q\lambda^{q-1} & \cdots & \mathrm{C}_q^{n-1}\lambda^{q-n+1}\\
0 & \lambda^q & \cdots & \mathrm{C}_q^{n-2}\lambda^{q-n+2}\\
\vdots &\vdots &  \ddots & \vdots \\
0 & 0 &\cdots &  \lambda^q
    \end{bmatrix}\\
 &=\begin{bmatrix}
     \alpha\lambda^{q}& \alpha\left(\lambda^q+q\lambda^{q-1}\right)  & \cdots  & \alpha\left (\lambda^q+q\lambda^{q-1}+\cdots+\mathrm{C}_q^{n-1}\lambda^{q-n+1}\right ) \\
     \vdots  & \vdots & \ddots  & \vdots \\
    \alpha\lambda^{q}& \alpha\left(\lambda^q+q\lambda^{q-1}\right)  & \cdots  & \alpha\left (\lambda^q+q\lambda^{q-1}+\cdots+\mathrm{C}_q^{n-1}\lambda^{q-n+1}\right )
    \end{bmatrix}
\end{aligned}
\end{equation*}
Hence, by direct computation, we have
\begin{equation}\label{dmeq3}
\|(J_n(\lambda))^pM(J_n(\lambda))^q\|_1=\alpha\theta_p\theta_q,
\end{equation}
where
\begin{equation*}
\begin{aligned}
\theta_q&=\lambda^{q}+\left(\lambda^q+q\lambda^{q-1}\right)+\cdots+\left (\lambda^q+q\lambda^{q-1}+\cdots+\mathrm{C}_q^{n-1}\lambda^{q-n+1}\right ), \\
\theta_p&=\lambda^p+\left(\lambda^p+p\lambda^{p-1}\right)+\cdots+ (\lambda^p+ p\lambda^{p-1}+\cdots+\mathrm{C}_p^{n-1}\lambda^{p-n+1}).
\end{aligned}
\end{equation*}
Since $\lambda<1$, we have
$$
\theta_q=n\lambda^{q}+(n-1)q\lambda^{q-1}+\cdots+2\mathrm{C}_q^{n-2}\lambda^{q-n+2}+\mathrm{C}_q^{n-1}\lambda^{q-n+1}\le n^2q^{n-1}\lambda^{q-n+1},
$$
which together with (\ref{dmeq3}) implies that
\begin{equation}\label{dmeq4}
\left\|X^{p}YX^{q}\right\|_1\leq \|(J_n(\lambda))^pM(J_n(\lambda))^q\|_1\le \alpha n^4(pq)^{n-1}\lambda^{p+q-2n+2}.
\end{equation}
It follows from (\ref{dmeq4}) with $p+q=k-1$ that
\begin{equation*}
\begin{aligned}
\left\| YX^{k-1}+XYX^{k-2}+\cdots +X^{k-1}Y\right \|_1&\leq \sum_{p=0}^{k-1}\alpha n^{4}\lambda^{-2n+1}p^{n-1}\left ( k-1-p\right )^{n-1}\lambda^{k}\\
    & \leq \alpha n^{4}\lambda ^{-2n+1}k^{2n}\lambda^{k},
    \end{aligned}
\end{equation*}
which together with
$$
\lim_{k\to\infty}k^{2n}\lambda^{k}=0,
$$
implies that the matrix $YX^{k-1}+XYX^{k-2}+\cdots +X^{k-1}Y$ converges to $O$ as $k\to\infty$. Since $\lambda<1$, by (\ref{realmeq1}), we have
$$
\lim_{k\to\infty}X^k=O,
$$
which together with (\ref{dmeq1}) yields that $A^{k}$ converges to $O$ as $k\to\infty$.

The necessity is obvious. In fact, if $A^{k}$ converges to $O$, then $A_{s}^{k}$ must converge to $O$, which together with (\ref{realmeq1}) means $\rho \left ( A_{s}\right )< 1$.
\end{proof}

\section{Convergence Analysis for the Collatz Method}\label{Colmet}

In this section, we provide the convergence analysis for the Collatz method given in \cite[Algorithm 1]{qcnew} for computing Perron eigenpair or Perron-Frobenius eigenpair of a dual number matrix with primitive or irreducible nonnegative standard part. As mentioned above, an irreducible nonnegative dual number matrix can be shifted to a primitive one via a positive parameter, so we only discuss the convergence of the Collatz method with respect to primitive dual number matrices.

Based on Theorem \ref{lemma2}, the Collatz method for computing Perron eigenpair of a primitive dual number matrix $A=A_{s}+A_{d}\varepsilon$ presented in \cite{qcnew} works as follows: choose
an arbitrary dual number vector $\mathbf x^{0}$ with a positive standard part, and let $\mathbf y^{0}=A\mathbf x^{0}$. For $k=0,1,2,\ldots $, compute
\begin{equation}\label{collatz1}
\mathbf x^{k+1}=\dfrac{\mathbf y^{k}}{\left \| \mathbf y^{k}\right \|_{2}},\qquad \mathbf y^{k+1}=A\mathbf x^{k+1},
\end{equation}
and
\begin{equation}\label{collatz2}
\underline{\lambda }_{k}=\underset{i}{\min}\dfrac{(A\mathbf{x}^k)_{i}}{x^k_{i}}, \qquad \overline{\lambda }_{k}=\underset{i}{\max}\dfrac{(A\mathbf{x}^k)_{i}}{x^k_{i}}.
\end{equation}

It was shown in \cite[Theorem 6.2]{qcnew} that the dual number sequences $\{\underline{\lambda}_{k}\}$ and $\{\overline{\lambda}_{k}\}$ generated by the Collatz method (\ref{collatz1}) and (\ref{collatz2})  satisfy
\begin{equation*}
\underline{\lambda }_{0}\leq \underline{\lambda }_{1}\leq \cdots \leq \underline{\lambda}_{k}\le \underline{\lambda}_{k+1}\le \cdots\le \lambda \leq \cdots\leq \overline{\lambda}_{k+1} \le \overline{\lambda}_{k}\le\cdots \leq \overline{\lambda }_{1}\leq \underline{\lambda }_{0},
\end{equation*}
where $\lambda$ is the Perron eigenvalue of $A$.

We wish to get the conclusion that both $\{\underline{\lambda}_{k}\}$ and $\{\overline{\lambda}_{k}\}$ converge to the same dual number $\lambda$. Some convergence and convergence rate results were observed
numerically in \cite{qcnew}. However, there was without theoretical proofs. Qi and Cui only proved that the standard parts of these two dual number sequences linearly converge to the same number, i.e.,
$\lim\limits_{k\to\infty}(\overline{\lambda}_{s,k}-\underline{\lambda}_{s,k})=0$ in
\cite[Theorem 6.3]{qcnew}. So, they proposed the second conjecture in \cite{qcnew} as follows:
\begin{conjecture}\label{conj2}
What is the condition such that both $\{\underline{\lambda}_{k}\}$ and $\{\overline{\lambda}_{k}\}$ converge to the same dual number $\lambda$ and what
is the convergence rate?
\end{conjecture}

We give a positive answer for Conjecture \ref{conj2} and show that both $\{\underline{\lambda}_{k}\}$ and $\{\overline{\lambda}_{k}\}$ generated  by the Collatz method (\ref{collatz1}) and (\ref{collatz2}) converge to the Perron eigenvalue $\lambda$ of the given primitive dual number matrix $A$ with the rate of $R$-linear convergence.

For this purpose, we need the following lemma.
\begin{lemma}\label{thm3}
Assume that the real numbers $\gamma>0$, $0<\eta<1$, $L>0$,  and $\lambda =\lambda_{s}+\lambda _{d}\varepsilon \in \mathbb{D}$ with $\lambda_{s}>0$, then for  sufficiently large $k$ and any
$t_i\in\mathbb{R}$ with $|t_i|\le L$ $(i=1,2,3,4)$, we have the absolute values of the standard part and the dual part of the following dual number
\begin{equation*}
\underset{\left |t_{i} \right |\leq L }{\max } \lambda  \frac{\gamma+t_{1}\eta^{k}+t_{2}\eta^{k}\varepsilon}{\gamma+t_{3}\eta^{k}+t_{4}\eta^{k}\varepsilon}
-\underset{\left |t_{i} \right |\leq L }{\min } \lambda  \frac{\gamma+t_{1}\eta^{k}+t_{2}\eta^{k}\varepsilon}{\gamma+t_{3}\eta^{k}+t_{4}\eta^{k}\varepsilon}
\end{equation*}
are no more than $\dfrac{8L\left ( 2\lambda _{s}+\left | \lambda _{d}\right |\right )}{\gamma}\eta^{k}$.
\end{lemma}
\begin{proof}
Since $\gamma>0$, $0<\eta<1$ and $\left |t_{3} \right |\leq L$, it holds for sufficiently large $k$ that \begin{equation}\label{dm1}
\gamma+t_{3}\eta^{k}>\dfrac{\gamma}{2}>0.
\end{equation}
Define
$$
\theta=\dfrac{\gamma+t_{1}\eta^{k}+t_{2}\eta^{k}\varepsilon}{\gamma+t_{3}\eta^{k}+t_{4}\eta^{k}\varepsilon}.
$$
By (\ref{dm1}), Definition \ref{def1} (ii) and Definition \ref{def2} (iii), the dual number $\lambda\theta$ can be written as
\begin{equation*}
\lambda \theta
=\frac{\lambda _{s}\left (\gamma+t_{1}\eta^{k} \right )}{\gamma+t_{3}\eta^{k}}+\left (\dfrac{\left (\gamma+ t_{1}\eta^{k}\right )\lambda _{d}+t_{2}\eta^{k}\lambda _{s}}{\gamma+t_{3}\eta^{k}} - \dfrac{\lambda _{s}\left (\gamma+t_{1}\eta^{k} \right )t_{4}\eta^{k}}{\left ( \gamma+t_{3}\eta^{k}\right )^{2}}\right )\varepsilon.
\end{equation*}
Hence, the maximal value  and the minimal value of the standard part of the dual number $\lambda\theta$ over the set $\Omega=\{t_i\in\mathbb{R}: ~|t_i|\le L, i=1,2,3,4\}$ are
\begin{equation}\label{dmmax}
\max_{t_i\in\Omega}  \dfrac{\lambda _{s}\left (\gamma+t_{1}\eta^{k} \right )}{\gamma+t_{3}\eta^{k}}=
\lambda _{s}+\max_{t_i\in\Omega} \dfrac{\left (t_{1}\eta^{k}-t_{3}\eta^{k} \right )\lambda _{s}}{\gamma+t_{3}\eta^{k}}\leq \lambda _{s}+\dfrac{4L\lambda _{s}}{\gamma}\eta^{k},
\end{equation}
and
\begin{equation}\label{dmmin}
\min_{t_i\in\Omega}  \dfrac{\lambda _{s}\left (\gamma+t_{1}\eta^{k} \right )}{\gamma+t_{3}\eta^{k}}=
\lambda _{s}+\min_{t_i\in\Omega} \dfrac{\left (t_{1}\eta^{k}-t_{3}\eta^{k} \right )\lambda _{s}}{\gamma+t_{3}\eta^{k}}\ge \lambda _{s}-\dfrac{4L\lambda _{s}}{\gamma}\eta^{k}.
\end{equation}
On the other hand, denote $(\lambda\theta)_d$ as the dual part of the dual number $\lambda\theta$, i.e.,
$$(\lambda\theta)_d=\dfrac{\left (\gamma+ t_{1}\eta^{k}\right )\lambda _{d}+t_{2}\eta^{k}\lambda _{s}}{\gamma+t_{3}\eta^{k}} - \dfrac{\lambda _{s}\left (\gamma+t_{1}\eta^{k} \right )t_{4}\eta^{k}}{\left ( \gamma+t_{3}\eta^{k}\right )^{2}}.$$
Hence, its maximal and minimal values over the set $\Omega$ are
\begin{equation}\label{dualmax}
\max_{t_i\in\Omega}(\lambda\theta)_d\leq \lambda _{d}+\dfrac{2L\left ( \lambda _{s}+2\left | \lambda _{d}\right |\right )}{\gamma}\eta^{k}+\dfrac{6L\lambda _{s}}{\gamma}\eta^{k}=\lambda _{d}+\dfrac{4L\left ( 2\lambda _{s}+\left | \lambda _{d}\right |\right )}{\gamma}\eta^{k},
\end{equation}
and
\begin{equation}\label{dualmin}
\min_{t_i\in\Omega}(\lambda\theta)_d\geq \lambda _{d}-\dfrac{2L\left ( \lambda _{s}+2\left | \lambda _{d}\right |\right )}{\gamma}\eta^{k}-\dfrac{6L\lambda _{s}}{\gamma}\eta^{k}=\lambda _{d}-\dfrac{4L\left ( 2\lambda _{s}+\left | \lambda _{d}\right |\right )}{\gamma}\eta^{k}.
\end{equation}
Combining (\ref{dmmax}), (\ref{dmmin}), (\ref{dualmax}) and (\ref{dualmin}), we obtain that the absolute values of the standard part and the dual part of
\begin{equation*}
\underset{\left |t_{i} \right |\leq L }{\max } \lambda  \frac{\gamma+t_{1}\eta^{k}+t_{2}\eta^{k}\varepsilon}{\gamma+t_{3}\eta^{k}+t_{4}\eta^{k}\varepsilon}
-\underset{\left |t_{i} \right |\leq L }{\min } \lambda  \frac{\gamma+t_{1}\eta^{k}+t_{2}\eta^{k}\varepsilon}{\gamma+t_{3}\eta^{k}+t_{4}\eta^{k}\varepsilon}
\end{equation*}
are no more than $\dfrac{8L\left ( 2\lambda _{s}+\left | \lambda _{d}\right |\right)}{\gamma}\eta^{k}$ for sufficiently large $k$.
\end{proof}

We give the main result which shows that the Collatz method given in \cite[Algorithm 1]{qcnew} has linear convergence.
\begin{theorem}\label{thm4}
Suppose that $A=A_{s}+A_{d}\varepsilon$ is an $n\times n$ primitive dual number matrix and its Perron eigenvalue is $\lambda=\lambda_{s}+\lambda_{d}\varepsilon\in\mathbb{D}_{++}$. Let $\left \{\underline{ \lambda}_{k} \right \}$ and $\left \{\overline{ \lambda}_{k} \right \}$ be the dual number sequences generated  by the Collatz method (\ref{collatz1}) and (\ref{collatz2}) from  an arbitrary initial iterative point $\mathbf x^{0}\in\mathbb{R}_{++}^n$. Then, we have
\begin{equation*}
\lim_{k\rightarrow \infty }\left ( \overline{ \lambda}_{k}-\underline{ \lambda}_{k}\right )=0.
\end{equation*}
Moreover, the convergence rate of the dual number sequence $\left \{  \overline{ \lambda}_{k}-\underline{ \lambda}_{k}  \right \}$ is $R$-linear, i.e., there exist constants $\tilde{\Delta}>0$ and $\eta\in (0,1)$ such that the absolute values of the standard part and the dual part of
$\overline{ \lambda}_{k}-\underline{ \lambda}_{k}$ are no more than $\tilde{\Delta}\eta^k$
for sufficiently large $k$.
\end{theorem}
\begin{proof}Since $A_s$ is a primitive real matrix and $\lambda_s$ is its Perron eigenvalue, we can assume that its Jordan decomposition is
$$A_{s}=P\text{diag}\left (\lambda_s,J_{1},J_{2},\cdots ,J_{t}\right )P^{-1},$$
where $J_{1},J_{2},\cdots ,J_{t}$ are Jordan blocks that belong to the eigenvalues $\lambda _{1},\lambda _{2},\cdots ,\lambda _{t}$, respectively.
Denote $B=P^{-1}A_{d}P=(b_{ij})$ and
$$\mathbf x_{s}=P
\begin{bmatrix}
    1\\
    0\\
    \vdots \\
    0
    \end{bmatrix},\qquad
\mathbf y_{s}=P^{-T}
\begin{bmatrix}
    1\\
    0\\
    \vdots \\
    0
\end{bmatrix}.$$
Then, we have
$$A_{s}\mathbf x_{s}=\lambda _{s}\mathbf x_{s},\quad \mathbf y_{s}^{T}A_{s}=\lambda _{s}\mathbf y_{s},\quad \mathbf y_{s}^{T}\mathbf x_{s}=1>0.$$
So, $\mathbf x_{s}$ and $\mathbf y_{s}$ are the right and left Perron eigenvectors of $A_{s}$, respectively. Moreover, by (\ref{eq5}), we also have
$$\lambda _{d}=\frac{\mathbf y_{s}^{T}A_{d}\mathbf x_{s}}{\mathbf y_{s}^{T}\mathbf x_{s}}=b_{11}.$$

Since the dual number matrix $\lambda^{-1}A$ is also primitive and its Perron eigenvalue is $1$, we consider its Jordan decomposition and still use the above notations. Without loss of generality, we may assume that
$$\lambda ^{-1}A=P{\rm diag}\left ( 1, J_{1},\cdots ,J_{t}\right )P^{-1}+PBP^{-1}\varepsilon ,$$
where $B=P^{-1}(\lambda^{-1}A)_{d}P=(b_{ij})$ and $J_{1},J_{2},\cdots ,J_{t}$ are Jordan blocks of $(\lambda ^{-1}A)_s$ that belong to the eigenvalues $\lambda _{1},\lambda _{2},\cdots ,\lambda _{t}$, respectively. Furthermore, $\left |\lambda _{i} \right |<1$ $(i=1,2,\ldots,t)$ and $b_{11}=0$.

Let $D={\rm diag}\left ( 1, J_{1},\cdots ,J_{t}\right )$. Similar to (\ref{dmeq1}), we have
\begin{equation}\label{lamda}
P^{-1}\left (\lambda ^{-1}A \right )^{k}P=D^{k}+\left ( BD^{k-1}+DBD^{k-2}+\cdots +D^{k-1}B\right )\varepsilon,
\end{equation}
where $D^l={\rm diag}\left ( 1, J^l_{1},\cdots ,J_{t}^l\right )$ for any positive integer $l$.
Since $\left |\lambda _{i} \right |<1$ $(i=1,2,\ldots,t)$, it follows from (\ref{dmeq2}) that $J_{i}^{k}\to O$ as $k\to\infty$.

According to the $t+1$ Jordan blocks, we divide the real matrix $B$ as follows:
$$B=
\begin{bmatrix}
    0 & B_{01} & \cdots  &B_{0 t} \\
    B_{10} & B_{11} & \cdots  & B_{1t}\\
    \vdots  & \vdots  & \ddots  &\vdots  \\
     B_{t0}& B_{t1} &\cdots   & B_{tt}
\end{bmatrix}.$$
Then, for any positive integers $u,v$, we have
\begin{equation*}
D^uBD^v=\begin{bmatrix}
        0 & B_{01}J_{1}^{v} & \cdots  &B_{0 t}J_{t}^{v} \\
        J_{1}^{u}B_{10} & J_{1}^{u}B_{11}J_{1}^{v} & \cdots  & J_{1}^{u}B_{1t}J_{t}^{v}\\
        \vdots  & \vdots  & \ddots  &\vdots  \\
        J_{t}^{u} B_{t0}& J_{t}^{u}B_{t1}J_{1}^{v}&\cdots   & J_{t}^{u}B_{tt}J_{t}^{v}
    \end{bmatrix}.
 \end{equation*}
 Hence, it follows from (\ref{lamda}) that
{\small \begin{equation*}
P^{-1}\left (\lambda ^{-1}A \right )^{k}P=D^{k}+\begin{bmatrix}
         0 & B_{01}\sum\limits_{u=0}^{k-1}J_{1}^{u}& \cdots  &B_{0 t}\sum\limits_{u=0}^{k-1}J_{t}^{u}\\
         \sum\limits_{u=0}^{k-1}J_{1}^{u}B_{10} & \sum\limits_{u=0}^{k-1}J_{1}^{u}B_{11}J_{1}^{k-1-u} & \cdots  & \sum\limits_{u=0}^{k-1}J_{1}^{u}B_{11}J_{t}^{k-1-u}\\
         \vdots  & \vdots  & \ddots  &\vdots  \\
         \sum\limits_{u=0}^{k-1}J_{t}^{u}B_{t0}& \sum\limits_{u=0}^{k-1}J_{t}^{u}B_{11}J_{1}^{k-1-u}&\cdots   & \sum\limits_{u=0}^{k-1}J_{t}^{u}B_{11}J_{t}^{k-1-u}
    \end{bmatrix}\varepsilon.
\end{equation*}}
For $i=1,2,\ldots,t$, assume that the dimension of  $J_i$ is $n_i$, then $n_1+\cdots+n_t=n-1$. By (\ref{dmeq2}), we have
\begin{equation*}
   \sum\limits_{u=0}^{k-1}J_{i}^{u}=
    \begin{bmatrix}
         \sum\limits_{u=0}^{k-1}\lambda _{i}^{u} & \cdots  &\sum\limits_{u=0}^{k-n_i}\mathrm{C}_{n_i-1+u}^{n_i-1}\lambda _{i}^{u} \\
         & \ddots  &\vdots  \\
         &  & \sum\limits_{u=0}^{k-1}\lambda _{i}^{u}
    \end{bmatrix}.
\end{equation*}
For $m=1,2,\ldots,n_i$, we have $$\mathrm{C}_{m-1+u}^{m-1}\lambda _{i}^{u}=\mathrm{C}_{m-1+u}^{m-1}\lambda _{i}^{\frac{u}{2}}\lambda _{i}^{\frac{u}{2}}.$$
Since $\mathrm{C}_{m-1+u}^{m-1}\lambda _{i}^{\frac{u}{2}}{\rightarrow}0$ as ${u\rightarrow \infty }$,  
and the series $\sum\limits_{u=0}^{+\infty }\lambda _{i}^{\frac{u}{2}}$ is convergent, so
$$\sum\limits_{u=0}^{k-m}\mathrm{C}_{m-1+u}^{m-1}\lambda _{i}^{u}$$ is also convergent as $k\to\infty$. Additionally,
similar to the proof of Theorem \ref{thm2}, we have
$$\lim_{k\to\infty}\sum\limits_{u=0}^{k-1}J_{i}^{u}B_{ij}J_{j}^{k-1-u}=O, \quad i,j=1,2,\ldots,t.
$$
so $\left (\lambda ^{-1}A \right )^{k}$ is convergent as $k\to\infty$. Hence, we can rewrite $\left (\lambda ^{-1}A \right )^{k}$ as $$\left (\lambda ^{-1}A \right )^{k}=\mathbf x_{s}\mathbf y_{s}^{T}+M_{k}\quad \text{with ~~$\lim_{k\to\infty}M_{k}=C\varepsilon$}$$ for some real matrix $C$.
Thus,
$$\dfrac{\left (\left (\lambda ^{-1}A \right )^{k+1}\mathbf x^{0} \right )_{i}}{\left (\left (\lambda ^{-1}A \right )^{k}\mathbf x^{0} \right )_{i}}=\dfrac{\left (\left (\mathbf x_{s}\mathbf y_{s}^{T}+M_{k+1} \right )\mathbf x^{0} \right )_{i}}{\left (\left (\mathbf x_{s}\mathbf y_{s}^{T}+M_{k} \right )\mathbf x^{0} \right )_{i}}\rightarrow 1~~~(k\to\infty),$$
which together with (\ref{collatz1}) and (\ref{collatz2}) means
$$\lim_{k\rightarrow \infty }\left ( \overline{ \lambda}_{k}-\underline{\lambda}_{k}\right )=0,$$ and hence, $$\lim_{k\rightarrow \infty } \overline{ \lambda}_{k}=\lim_{k\rightarrow \infty }\underline{\lambda}_{k}=\lambda.$$
This shows that the Collatz method is convergent for primitive dual number matrices.

Next, we will provide the convergence rate of the Collatz method. Denote $\sigma=\underset{1\le i\le t}{\max} \left | \lambda_{i}\right |$. Clearly, $\sigma<1$. It follows from (\ref{dmeq2}) that
\begin{equation*}
J_{i}^{k}=\begin{bmatrix}
\lambda_i^k & k\lambda_i^{k-1} & \cdots & \mathrm{C}_k^{n_i-1}\lambda_i^{k-n_i+1}\\
0 & \lambda_i^k & \cdots & \mathrm{C}_k^{n_i-2}\lambda_i^{k-n_i+2}\\
\vdots &\vdots &  \ddots & \vdots \\
0 & 0 &\cdots &  \lambda_i^k
\end{bmatrix},\quad  i=1,2,\ldots,t.
\end{equation*}
By direct computations, the $(v,w)$-th element of $J_{i}^{k}$ satisfies
$$\left |\left (J_{i}^{k} \right )_{vw} \right |\leq \sigma^{1-n_i}k^{n_i}\sigma^{k}.$$
Hence, there exists sufficiently large $k_1$ such that
\begin{equation}\label{eqk1}
\left |\left (J_{i}^{k} \right )_{vw} \right |\leq \alpha ^{\frac{k}{2}},\quad \forall k\ge k_1.
\end{equation}
Similarly, since
\begin{equation*}
    \begin{aligned}
\left |\left (\sum_{u=k-n_i+1}^{+\infty } J_{i}^{u}\right )_{vw} \right | \leq  \underset{1\leq m \leq n_i}{\max} \left \{ \sum_{u=k}^{+\infty } \mathrm{C}_{u}^{m-1}\sigma ^{u-m+1}  \right \} \leq \sigma ^{1-n_i} \sum_{u=k}^{+\infty }u^{n_i}\sigma ^{u},
\end{aligned}
\end{equation*}
thus, there exists $k_{2}\in \mathbb{N}$ such that when $k\ge k_2$ we have
\begin{equation}\label{eqk2}
\left |\left (  B_{0,i} \sum_{u=k-n_i+1}^{+\infty } J_{i}^{u}\right )_{vw} \right | \leq  \sigma^{\frac{k}{2}}, \quad \left |\left ( \sum_{u=k-n_i+1}^{+\infty } J_{i}^{u}B_{i,0}\right )_{vw} \right | \leq  \sigma^{\frac{k}{2}}.
\end{equation}
Additionally, similar to the proof of Theorem \ref{thm2}, we have from (\ref{dmeq4}) that
$$\left |\left (\sum_{u=0}^{k-1}J_{i}^{u}B_{ij}J_{j}^{k-1-u} \right )_{vw} \right |\leq \alpha n^{4}\sigma ^{-2n+1}k^{2n}\sigma^{k},$$
where $\alpha=\max_{1\le i,j\le t}\|B_{ij}\|_1$.
Thus, there exists  $k_{3}\in \mathbb{N}$ such that when $k\geq k_{3}$,
\begin{equation}\label{eqk3}
\left |\left (\sum_{u=0}^{k-1}J_{i}^{u}B_{ij}J_{j}^{k-1-u} \right )_{vw} \right |\leq \sigma^{\frac{k}{2}}.
\end{equation}

Let $\eta=\sqrt{\sigma}$ and $R_k=(\lambda^{-1}A)^{k}-\left (\mathbf x_{s}\mathbf y_{s}^{T}+C\varepsilon  \right )$. Then, combining (\ref{eqk1}), (\ref{eqk2}) and (\ref{eqk3}), there exists a constant $L>0$ such that
\begin{equation}\label{eqk123}
\left|(R_k)_{s,vw}\right |\leq L\eta^{k},\quad \left|(R_k)_{d,vw}\right |\leq L\eta^{k}, \quad \forall k\ge \max\{k_1,k_2,k_3\}.
\end{equation}
For $i=1,\ldots,n$, we have
\begin{equation}\label{eqsd}
\begin{aligned}
\dfrac{\left (A\left (\lambda ^{-1}A \right )^{k}\mathbf x^{0} \right )_{i}}{\left (\left (\lambda ^{-1}A \right )^{k}\mathbf x^{0} \right )_{i}}&=\lambda\dfrac{\left (\left (\lambda ^{-1}A \right )^{k+1}\mathbf x^{0} \right )_{i}}{\left (\left (\lambda ^{-1}A \right )^{k}\mathbf x^{0} \right )_{i}}\\
&=\lambda\dfrac{\left ((\mathbf x_{s}\mathbf y_{s}^{T}+C\varepsilon+R_{k+1})\mathbf x^{0} \right )_{i}}{\left ((\mathbf x_{s}\mathbf y_{s}^{T}+C\varepsilon+R_{k})\mathbf x^{0} \right )_{i}}\\
&=\lambda\dfrac{(\mathbf x_{s}\mathbf y_{s}^{T}\mathbf x^{0}+C\mathbf x^{0}\varepsilon)_i+(R_{k+1}\mathbf x^{0})_{i}}{(\mathbf x_{s}\mathbf y_{s}^{T}\mathbf x^{0}+C\mathbf x^{0}\varepsilon)_i+(R_{k}\mathbf x^{0})_{i}}\\
&=\lambda\dfrac{1+\dfrac{(R_{k+1}\mathbf x^{0})_{i}}{(\mathbf x_{s}\mathbf y_{s}^{T}\mathbf x^{0}+C\mathbf x^{0}\varepsilon)_i}}{1+\dfrac{(R_{k}\mathbf x^{0})_{i}}{(\mathbf x_{s}\mathbf y_{s}^{T}\mathbf x^{0}+C\mathbf x^{0}\varepsilon)_i}},
\end{aligned}
\end{equation}
and by Definition \ref{def2} (iii),
\begin{equation*}
\begin{aligned}
\dfrac{(R_{k}\mathbf x^{0})_{i}}{(\mathbf x_{s}\mathbf y_{s}^{T}\mathbf x^{0}+C\mathbf x^{0}\varepsilon)_i}&=\dfrac{(R_{k}\mathbf x^{0})_{i,s}+(R_{k}\mathbf x^{0})_{i,d}\varepsilon}{\mathbf y_{s}^{T}\mathbf x^{0}(\mathbf x_{s})_i+(C\mathbf x^{0})_i\varepsilon}\\
&=\dfrac{(R_{k}\mathbf x^{0})_{i,s}}{\mathbf y_{s}^{T}\mathbf x^{0}(\mathbf x_{s})_i}+\left(\dfrac{(R_{k}\mathbf x^{0})_{i,d}}{\mathbf y_{s}^{T}\mathbf x^{0}(\mathbf x_{s})_i}-\dfrac{(R_{k}\mathbf x^{0})_{i,s}(C\mathbf x^{0})_i}{(\mathbf y_{s}^{T}\mathbf x^{0}(\mathbf x_{s})_i)^2}\right)\varepsilon.
\end{aligned}
\end{equation*}
We consider the standard part and dual part of the above dual number. Note that $\mathbf x^{0}, \mathbf x_{s}, \mathbf y_{s}\in\mathbb{R}_{++}^n$ and denote $\delta=\mathbf y_{s}^{T}\mathbf x^{0}$, $\mu=\min_{1\le i\le n}\{(\mathbf x_{s})_i\}$. Then, it follows from (\ref{eqk123}) that
\begin{equation}\label{eqlem1}
\left|\dfrac{(R_{k}\mathbf x^{0})_{i,s}}{\mathbf y_{s}^{T}\mathbf x^{0}(\mathbf x_{s})_i}\right|=\left|\dfrac{\left((R_{k})_s\mathbf x^{0}\right)_{i}}{\delta(\mathbf x_{s})_i}\right|\le \dfrac{L\|\mathbf x^0\|_1}{\delta\mu}\eta^k,\quad \forall k\ge \max\{k_1,k_2,k_3\},
\end{equation}
and when $k\ge \max\{k_1,k_2,k_3\}$, it also holds
\begin{equation}\label{eqlem2}
\left|\dfrac{(R_{k}\mathbf x^{0})_{i,d}}{\mathbf y_{s}^{T}\mathbf x^{0}(\mathbf x_{s})_i}-\dfrac{(R_{k}\mathbf x^{0})_{i,s}(C\mathbf x^{0})_i}{(\mathbf y_{s}^{T}\mathbf x^{0}(\mathbf x_{s})_i)^2}\right|\le \left(\dfrac{L\|\mathbf x^0\|_1}{\delta\mu}+\dfrac{L\|\mathbf x^0\|^2_1\|C\|_{\infty}}{(\delta\mu)^2}\right)\eta^k.
\end{equation}
Let 
$$\Delta=\dfrac{L\|\mathbf x^0\|_1}{\delta\mu}+\dfrac{L\|\mathbf x^0\|^2_1\|C\|_{\infty}}{(\delta\mu)^2}.
$$
Then, in light of (\ref{eqlem1}) and (\ref{eqlem2}), the absolute values of the standard part and dual part of $$\dfrac{(R_{k}\mathbf x^{0})_{i}}{(\mathbf x_{s}\mathbf y_{s}^{T}\mathbf x^{0}+C\mathbf x^{0}\varepsilon)_i}$$ are no more that $\Delta\eta^k$ when $k\ge \max\{k_1,k_2,k_3\}$.

On the other hand, it follows from (\ref{eqk123}) that  for $k\ge \max\{k_1,k_2,k_3\}$,
\begin{equation*}
\left|(R_{k+1})_{s,vw}\right |\leq L\eta^{k+1}<L\eta^{k},\quad \left|(R_k)_{d,vw}\right |\leq L\eta^{k+1}<L\eta^{k}.
\end{equation*}
Hence,  similar to the proof of (\ref{eqlem1}) and (\ref{eqlem2}), we can obtain that the absolute values of the standard part and dual part of $$\dfrac{(R_{k+1}\mathbf x^{0})_{i}}{(\mathbf x_{s}\mathbf y_{s}^{T}\mathbf x^{0}+C\mathbf x^{0}\varepsilon)_i}$$ are no more that $\Delta\eta^k$ when $k\ge \max\{k_1,k_2,k_3\}$.

Hence, by Lemma \ref{thm3} and (\ref{eqsd}), we know that the absolute values of the standard part and dual part of
\begin{equation*}
\underset{i}{\max }\dfrac{\left (A\left (\lambda ^{-1}A \right )^{k}\mathbf x^{0} \right )_{i}}{\left (\left (\lambda ^{-1}A \right )^{k}\mathbf x^{0} \right )_{i}}-\underset{i}{\min }\dfrac{\left (A\left (\lambda ^{-1}A \right )^{k}\mathbf x^{0} \right )_{i}}{\left (\left (\lambda ^{-1}A \right )^{k}\mathbf x^{0} \right )_{i}}
\end{equation*}
are no more than
$$8\Delta\left ( 2\lambda _{s}+\left | \lambda _{d}\right |\right)\eta^{k}
$$when $k\ge \max\{k_1,k_2,k_3\}$. Let $\tilde{\Delta}=8\Delta\left ( 2\lambda _{s}+\left | \lambda _{d}\right |\right)$. Then,  the absolute values of the standard part and the dual part of
$\overline{ \lambda}_{k}-\underline{ \lambda}_{k}$ are no more than $\tilde{\Delta}\eta^k$
for sufficiently large $k$. Hence, the convergence rate of the dual number sequence $\left \{  \overline{ \lambda}_{k}-\underline{ \lambda}_{k}  \right \}$ is $R$-linear.
\end{proof}

The above theorem shows that if the initial point $\mathbf{x}^0$ is a positive real vector then the Collatz method (\ref{collatz1}) and (\ref{collatz2}) $R$-linearly converges to the Perron eigenvalue $\lambda$ of a primitive dual number matrix $A$ with an explicit rate $\eta=\sqrt{\sigma}$ where $\sigma=\underset{1\le i\le t}{\max} \left | \lambda_{i}\right |$ and $\lambda_{i} (i=1,\ldots,t)$ are other eigenvalues except for the Perron eigenvalue $1$ of $\lambda^{-1}A$.

\section{Conclusion}\label{conclusion}
Qi and Cui \cite{qcnew} extended Perron-Frobenius theorem for nonnegative matrices to dual number matrices with primitive or irreducible nonnegative standard parts and proposed the Collatz method to compute their Perron eigenvalues. However, they have not obtained the theoretical proofs for its convergence. In this paper, we first show the necessary and sufficient condition for convergence of $k$-order power of a dual number matrix (Theorem \ref{thm2}). Based on Theorem \ref{thm2} and Jordan decomposition of primitive real matrices, we show that the dual number sequences generated by the Collatz method (\ref{collatz1}) and (\ref{collatz2}) converges to the Perron eigenpair under the condition of a positive real vector as initial iterative point. An explicit $R$-linear rate of the Collatz method  is given (Theorem \ref{thm4}).


\section*{References}

\bibliographystyle{elsarticle-harv}

\begin{thebibliography}{99}

\bibitem{qc1} J. Angeles, The dual generalized inverses and their applications in kinematic
synthesis, in: Latest Advances in Robot Kinematics, Springer, Dordrecht,
2012, pp. 1--12.

\bibitem{qc2} A. Berman and R.J. Plemmons, Nonnegative Matrices in the Mathematical
Sciences, SIAM, Philadelphia, 1994.

\bibitem{qc4} C. Cui and L. Qi, A power method for computing the dominant eigenvalue
of a dual quaternion Hermitian matrix, April 2023, arXiv:2304.04355.

\bibitem{qc5} Y.L. Gu and L. Luh, Dual-number transformation and its applications to
robotics, IEEE J. Robot. Autom. 3, (1987) 615--623.

\bibitem{qc9} E. Pennestri and P.P. Valentini, Linear dual algebra algorithms and their
applications to kinematics, in: Multibody Dynamics, Springer, Dordrecht,
2009, pp. 207--229.

\bibitem{qc10} E. Pennestri, P.P. Valentini, D. De Falco and J. Angeles, Dual Cayley-
Klein parameters and M\"{o}bius transform: Theory and applications,
Mech. Mach. Theory 106, (2016) 50--67.

\bibitem{qc14} L. Qi and Z. Luo, Eigenvalues and singular value decomposition of dual
complex matrices, October 2021, arXiv:2110.02050.

\bibitem{qc12} L. Qi and C. Cui, Eigenvalues and Jordan forms of dual complex matrices,
June 2023, arXiv: 2306.12428v2.

\bibitem{qcnew} L. Qi and C. Cui, Dual number matrices with primitive and irreducible
nonnegative standard parts, July 2023, arXiv:2306.16140.

\bibitem{qc11} L. Qi, D.M. Alexander, Z. Chen, C. Ling and Z. Luo, Low rank approximation
of dual complex matrices, March 2022, arXiv:2201.12781.

\bibitem{qc13} L. Qi, C. Ling and H. Yan, Dual quaternions and dual quaternion vectors,
Comm. Appl.  Math. Comput. 4, (2022) 1494--1508.

\bibitem{qc16} R. Varga, Matrix Iterative Analysis, Prentice-Hall, Englewood Cliffs, NJ,
1962.

\bibitem{qc17} H. Wang, Characterization and properties of the MPDGI and DMPGI,
Mech. Mach. Theory 158, (2021) 104212.

\bibitem{qc18} H. Wang, C. Cui and Y. Wei, The QLY least-squares and the QLY leastsquares
minimal-norm of linear dual least squares problems, Linear Multi. Algebra, DOI:10.1080/03081087.2023.2223348.




\end{thebibliography}


\end{document}